\documentclass[graybox]{svmult}

\usepackage{type1cm}        
\usepackage{makeidx}         
\usepackage{graphicx}        
\usepackage{multicol}        
\usepackage[bottom]{footmisc}

\usepackage{newtxtext}       %
\usepackage{newtxmath}       

\makeindex             

\newcommand{\Lsp}[1]{\operatorname{L^{#1}}}
\DeclareMathOperator{\dif}{d}
\begin{document}

\title*{Green's functions for even order boundary value problems}
\author{Alberto Cabada and Luc{\' i}a L\'opez-Somoza}
\institute{Departamento de Estat\'istica, An\'alise Matem\'atica e Optimizaci\'on, 
	Instituto de Matem\'aticas, Facultade de Matem\'aticas,
	Universidade de Santiago de Compostela. \email{alberto.cabada@usc.es; lucia.lopez.somoza@usc.es}
}
%
%
\maketitle

\abstract{In this paper we will show several properties of the Green's functions related to various boundary value problems of arbitrary even order. In particular, we will write the expression of the Green's functions related to the general differential operator of order $2n$ coupled to Neumann, Dirichlet and mixed boundary conditions, as a linear combination of the Green's functions corresponding to periodic conditions on a different interval. This will allow us to ensure the constant sign of various Green's functions and to deduce spectral results. }

\section{Introduction}
Boundary value problems have been widely studied. This is due to the fact that these problems arise in many areas to model from most of physical problems to biological or economical ones.

It is very well-known that the solutions of a given boundary value problem coincide with fixed points of related integral operators which have as kernel the associated Green's function in each case. So, the Green's functions play a very important role in the study of boundary value problems. 

In particular, some of the main techniques applied to prove the existence of solutions of nonlinear boundary value problems are, among others, monotone iterative techniques  (see \cite{HeikLaks,LaddeLaksVat,TorresZhang2003}), lower and upper solutions method (see \cite{Cab3, DeC}) or fixed points theorems (see \cite{HeikLaks,Tor}). In all these cases, the constant sign of the associated Green's functions is usually fundamental to prove such results. 

Traditionally, the most studied boundary value problems have been the periodic and the two-point ones. In this paper we will take advantage of such studies by finding some connections between Green's functions of various separated two point boundary conditions and Green's functions of periodic problem. The key idea is that the expression of the Green's function related to each two-points case can be obtained as a linear combination of the Green's function of periodic problems.

From these expressions relating the different Green's functions, we will be able to compare their constant sign. These results will allow us to obtain some comparison principles which assure that, under certain hypotheses, the solution of a boundary value problem under some suitable conditions is bigger in every point than the solution of the same equation under another type of boundary conditions.

We will also obtain a decomposition of the spectrum of some problems as a combination of the other ones and some relations of order between the first eigenvalues of the considered problems.

The paper is organized as follows: Section 2 includes some preliminary results and proves a symmetry property which will be satisfied by some Green's functions. In Section 3, the aforementioned decomposition of Green's functions is fully detailed. Section 4 proves some results relating the constant sign of Green's function and includes also some counterexamples, showing that some properties which hold for second order boundary value problems are not true for higher order ones. In Section 5, both the spectra and the first eigenvalues of the considered problems are related. Section 6 proves some point-by-point relations between different Green's functions and also between solutions of some linear problem under several boundary conditions. 

It must be pointed out that the study developed in Sections 3 to 6 has been done in \cite{CabCidSom2016} for the particular case of Hill's equation. This is generalized here for any $2n$-th order boundary value problem. However, some arguments which worked for Hill's equation (mainly the ones related with oscillation theory) do not hold for $n>1$. This implies that some of the results proved in this paper will be weaker than the corresponding ones for Hill's equation (in particular, as we have mentioned, some counterexamples will be shown in Section 4).

\section{Preliminary results}
Consider the $2n$-order general linear operator
\begin{equation}
	\label{e-L-2n}
	L\,u(t)\equiv u^{(2n)}(t)+a_{2n-1}(t)\,u^{(2n-1)}(t)+\dots +a_1(t)\,u'(t)+a_0(t)\,u(t), \quad t \in I,
\end{equation}
with $I\equiv [0,T]$, $a_k: I \to \mathbb{R}, \ a_k\in \Lsp{\alpha}(I),\; \alpha\geq 1$, $k=0,\dots,2n-1$.  

We will work with the space 
\[W^{2n,1}(I)=\left\{u\in\mathcal{C}^{2n-1}(I): \ u^{(2n-1)}\in \mathcal{A}\mathcal{C}(I) \right\},\]
where $\mathcal{A}\mathcal{C}(I)$ denotes the set of absolutely continuous functions on $I$. In particular, we will consider $X\subset W^{2n,1}(I)$ a Banach space such that the following definition is satisfied.
\begin{definition} Given a Banach space $X$, operator $L$ is said to be nonresonant in $X$ if and only if the homogeneous equation 
	\begin{equation*}
		L\,u(t)=0\quad \mbox{ a. e. }\; t\in I,\qquad u\in X,
	\end{equation*}
	has only the trivial solution.
\end{definition}

It is very well known that if $\sigma \in \Lsp{1}(I)$ and operator $L$ is nonresonant in $X$, then the nonhomogeneous problem
\[L\,u(t)=\sigma(t)\quad \mbox{ a.\,e.  }t\in I,\quad u\in X,\]
has a unique solution given by
\[u(t)=\int_0^T G[T](t,s)\,\sigma(s)\,\dif s,\qquad \forall\; t\in I,\]
where $G[T]$ denotes the Green's function related to operator $L$ on $X$ and it is uniquely determined. See \cite{Cablibro} for details.
%
%
%
%

We will introduce now an auxiliary linear operator, whose coefficients will be defined from the ones of operator $L$ as follows:
\begin{equation*}
	\widetilde{L}\,u(t)\equiv  u^{(2n)}(t)+\hat{a}_{2n-1}(t)\,u^{(2n-1)}(t)+\tilde{a}_{2n-2}(t)\,u^{(2n-2)}(t)+\dots +\hat{a}_1(t)\,u'(t)+\tilde{a}_0(t)\,u(t), \quad t \in J\equiv[0,2T],
\end{equation*}
where $\tilde{a}_{2k}$ and $\hat{a}_{2k+1}$, $k=0,\dots,n-1$, are the even and odd extensions of $a_{2k}$ and $a_{2k+1}$ to $J$.

We obtain the following symmetric property for Green's functions related to operator $\widetilde{L}$.
\begin{lemma}\label{l-sim-G-2n}
	Let $X\subset W^{2n,1}(J)$ be a Banach space such that operator $\widetilde{L}$ is nonresonant in $X$ and let $G[2T]$ denote the corresponding Green's function. Moreover, suppose that if $v\in X$ and $w(t):=v(2\,T-t)$, $t\in J$, then $w\in X$. Then the following equality holds:
	\begin{equation}\label{e-l-sim-G-2n}
		G[2\,T](t,s)=G[2\,T](2\,T-t,2\,T-s)\qquad \forall \;(t,s)\in J \times J.
	\end{equation}
\end{lemma}

\begin{proof}
	Let $\tilde \sigma\in \Lsp{1}(J)$ be arbitrarily chosen and consider the problem 
	\[\widetilde{L} v(t)=\widetilde{\sigma}(t), \ \text{a.\,e. } t\in J, \ v\in X. \]
	
	Since operator $\widetilde{L}$ is nonresonant in $X$, this problem has a unique solution $v$ which is given by
	\[ v(t)=\int_0^{2\,T}G[2\,T](t,s)\,\tilde \sigma(s)\,\dif s. \]
	
	On the other hand, due to the fact that $\tilde a_{2k}(t)=\tilde a_{2k}(2\,T-t)$ and $\hat a_{2k+1}(t)=-\hat a_{2k+1}(2\,T-t)$, it is easy to verify that $w(t)=v(2\,T-t)$ is the unique solution of the problem
	\[\widetilde{L} w(t)=\widetilde{\sigma}(2\,T-t), \ \text{a.\,e. } t\in J, \ w\in X. \]
	
	Therefore, 
	\[ w(t)=\int_0^{2\,T}G[2\,T](t,s)\,\tilde \sigma(2\,T-s)\,\dif s=\int_0^{2\,T}G[2\,T](t,2\,T-s)\,\tilde \sigma(s)\,\dif s.\]
	
	Now, since
	\[ w(t)= v(2\,T-t)= \int_0^{2\,T}G[2\,T](2\,T-t,s)\,\tilde \sigma(s)\,\dif s, \]
	and $\tilde \sigma\in \Lsp{1}(J)$ is arbitrary, we arrive at the following equality
	\[G[2\,T](2\,T-t,s)=G[2\,T](t,2\,T-s)\qquad \forall \;(t,s)\in J\times J\]
	or, which is the same,
	\begin{equation*}
		G[2\,T](t,s)=G[2\,T](2\,T-t,2\,T-s)\qquad \forall \;(t,s)\in J \times J.
	\end{equation*}
\end{proof}

In addition, we will consider another auxiliary operator $\widetilde{\widetilde{L}}$ which will be constructed from $\widetilde{L}$ in the same way than $\widetilde{L}$ has been constructed from $L$, that is:
\begin{equation*}
	\widetilde{\widetilde{L}}\,u(t)\equiv  u^{(2n)}(t)+\hat{\hat{a}}_{2n-1}(t)\,u^{(2n-1)}(t) +\tilde{\tilde{a}}_{2n-2}(t)\,u^{(2n-2)}(t)+\dots +\hat{\hat{a}}_1(t)\,u'(t)+\tilde{\tilde{a}}_0(t)\,u(t),
\end{equation*}
$t \in [0,4T]$, where $\tilde{\tilde{a}}_{2k}$ and $\hat{\hat{a}}_{2k+1}$, $k=0,\dots,n-1$, are the even and odd extensions to the interval $[0,4T]$ of $\tilde{a}_{2k}$ and $\hat{a}_{2k+1}$, respectively.

To finish with this preliminary section, we will show two particular cases of some more general spectral results given in \cite[Lemmas 1.8.25 and 1.8.33]{Cablibro}. For these results we need to introduce a new differential operator.

For any $\lambda\in \mathbb{R}$, consider operator $L[\lambda]$ defined from operator $L$ given by
\begin{equation*}
	L[\lambda]\,u(t)\equiv u^{(2n)}(t)+a_{2n-1}(t)\,u^{(2n-1)}(t)+\dots +a_1(t)\,u'(t)+(a_0(t)+\lambda)\,u(t), \quad t \in I.
\end{equation*}
In particular, note that $L\equiv L[0]$. When working with this operator, to stress the dependence of the Green's function on the parameter $\lambda$, we will denote by $G[\lambda,T]$ the Green's function related to $L[\lambda]$. Again, note that $G[T]\equiv G[0,T]$. Analogous notation can we used for $\widetilde{L}[\lambda]$ and $\widetilde{\widetilde{L}}[\lambda]$, whose related Green's functions will be denoted by $G[\lambda,2\,T]$ and $G[\lambda,4\,T]$, respectively. 

\begin{lemma}\label{l-M-NT-2n}
	Suppose that operator $L$ is nonresonant in a Banach space $X$, its related Green's function $G[T]$ is nonpositive on $I \times I$, and satisfies condition
	\begin{enumerate}
		\item[$(N_g)$] There is a continuous function $\phi(t) >0$ for all $t \in (0,T)$ and $k_1, \; k_2 \in \Lsp{1}(I)$, such that $k_1(s)<k_2(s)<0$ for a.\,e. $s \in I$, satisfying
		\[ \phi(t)\, k_1(s) \le  G[T](t,s) \le \phi(t)\, k_2(s), \quad \mbox{for a.\,e. } (t,s)  \in I \times I.\]
	\end{enumerate}
	
	Then $G[\lambda,T]$ is nonpositive on $I \times I$ if and only if $\lambda\in (-\infty, \lambda_1(T))$ or  $\lambda\in [-\bar{\mu}(T), \lambda_1(T))$, with $\lambda_1(T) >0$ the first eigenvalue of operator $L$ in $X$ and ${\bar{\mu}(T) \ge 0}$ such that $L[-\bar{\mu}(T)]$ is nonresonant in $X$ and the related nonpositive Green's function $G[-\bar{\mu}(T),T]$ vanishes at some point of the square $I\times I$.
\end{lemma}

\begin{lemma}
	\label{l-M-PT-2n}
	Suppose that operator $L$ is nonresonant in a Banach space $X$, its related Green's function $G[T]$ is nonnegative on $I \times I$, and satisfies condition
	\begin{enumerate}
		\item[$(P_g)$] There is a continuous function $\phi(t) >0$ for all $t \in (0,T)$ and $k_1, \; k_2 \in \Lsp{1}(I)$, such that $0<k_1(s)<k_2(s)$ for a.\,e. $s \in I$, satisfying
		\[\phi(t)\, k_1(s) \le  G[T](t,s) \le \phi(t)\, k_2(s), \quad \mbox{for a.\,e. } (t,s)  \in I \times I.\]
	\end{enumerate}
	
	Then $G[\lambda,T]$ is nonnegative on $I \times I$ if and only if $\lambda\in (\lambda_1(T),\infty)$ or  $\lambda\in (\lambda_1(T), \bar{\mu}(T)]$, with $\lambda_1(T) <0$ the first eigenvalue of operator $L$ in $X$ and $\bar{\mu}(T) \ge 0$ such that $L[\bar{\mu}(T)]$ is nonresonant in $X$ and the related nonnegative Green's function $G[\bar{\mu}(T),T]$ vanishes at some point of the square $I\times I$.
\end{lemma}


\section{Decomposing Green's functions}\label{S-Decomp-Green}
In this section we will obtain the expression of the Green's function of different two-point boundary value problems (Neumann, Dirichlet and Mixed problems) as a sum of Green's functions of other related problems. 

A similar decomposition has been made in \cite{CabCidSom2016} for the particular case of $n=1$ and $a_1\equiv 0$, which will be generalised here. There, the authors worked with Hill's operator, namely
\[\mathcal{L}\,u(t)\equiv u''(t)+a(t)\,u(t), \]
with $a\in \Lsp{\alpha}(I)$, $\alpha\ge 1$.

In this case, we will deal with some problems related to operators $L$ and $\widetilde{L}$ and the periodic problem related to $\widetilde{\widetilde{L}}$, which we describe in the sequel:
\begin{itemize}
	\item Neumann problem on the interval $[0,T]$:
	\vspace*{-5pt}
	\begin{equation}\label{e-N-2n}\tag{$N,\,T$}\left\{\begin{array}{l}
			L\,u(t)=\sigma(t),\quad \mbox{ a. e.  }\;t\in I,\\[5pt]
			u^{(2k+1)}(0)=u^{(2k+1)}(T)=0, \ k=0,\dots,n-1.
		\end{array}\right.\end{equation}
		\vspace*{-5pt}
	\item Dirichlet problem on the interval $[0,T]$:
		\vspace*{-5pt}
	\begin{equation}\label{e-D-2n}\tag{$D,\,T$}\left\{\begin{array}{l}
			L\,u(t)=\sigma(t),\quad \mbox{ a. e.  }\;t\in I,\\[5pt]
			u^{(2k)}(0)=u^{(2k)}(T)=0, \ k=0,\dots,n-1.
		\end{array}\right.\end{equation}
		\vspace*{-5pt}
	\item Mixed problem 1 on the interval $[0,T]$:
		\vspace*{-5pt}
	\begin{equation}\label{e-M1-2n}\tag{$M_1,\,T$}\left\{\begin{array}{l}
			L\,u(t)=\sigma(t),\quad \mbox{ a. e.  }\;t\in I,\\[5pt]
			u^{(2k+1)}(0)=u^{(2k)}(T)=0, \ k=0,\dots,n-1.
		\end{array}\right.\end{equation}
		\vspace*{-5pt}
	\item Mixed problem 2 on the interval $[0,T]$:
		\vspace*{-5pt}
	\begin{equation}\label{e-M2-2n}\tag{$M_2,\,T$}\left\{\begin{array}{l}
			L\,u(t)=\sigma(t),\quad \mbox{ a. e.  }\;t\in I,\\[5pt]
			u^{(2k)}(0)=u^{(2k+1)}(T)=0, \ k=0,\dots,n-1.
		\end{array}\right.\end{equation}
		\vspace*{-5pt}
	\item Periodic problem on the interval $[0,2\,T]$:
		\vspace*{-5pt}
	\begin{equation}\label{e-P-2T-2n}\tag{$P,\,2\,T$} \left\{\begin{array}{l} \widetilde{L}\,u(t)=\tilde \sigma(t),\quad \mbox{ a. e.   }\; t\in J,\\[5pt]
			u^{(k)}(0)=u^{(k)}(2\,T),\ k=0,\dots,2n-1.
		\end{array}\right.\end{equation}
		\vspace*{-5pt}
	\item Antiperiodic problem on the interval $[0,2\,T]$:
		\vspace*{-5pt}
	\begin{equation}\label{e-A-2T-2n}\tag{$A,\,2\,T$} \left\{\begin{array}{l} \widetilde{L}\,u(t)=\tilde \sigma(t),\quad \mbox{ a. e.   }\; t\in J,\\[5pt]
			u^{(k)}(0)=-u^{(k)}(2\,T),\ k=0,\dots,2n-1.
		\end{array}\right.\end{equation}
		\vspace*{-5pt}
	\item Neumann problem on the interval $[0,2\,T]$:
		\vspace*{-5pt}
	\begin{equation}\label{e-N-2T-2n}\tag{$N,\,2\,T$} \left\{\begin{array}{l} \widetilde{L}\,u(t)=\tilde \sigma(t),\quad \mbox{ a. e.   }\; t\in J,\\[5pt]
			u^{(2\,k+1)}(0)=u^{(2\,k+1)}(2\,T)=0,\ k=0,\dots,n-1.
		\end{array}\right.\end{equation}
		\vspace*{-5pt}
	\item Dirichlet problem on the interval $[0,2\,T]$:
		\vspace*{-5pt}
	\begin{equation}\label{e-D-2T-2n}\tag{$D,\,2\,T$} \left\{\begin{array}{l} \widetilde{L}\,u(t)=\tilde \sigma(t),\quad \mbox{ a. e.   }\; t\in J,\\[5pt]
			u^{(2\,k)}(0)=u^{(2\,k)}(2\,T)=0,\ k=0,\dots,n-1.
		\end{array}\right.\end{equation}
		\vspace*{-5pt}
	\item Periodic problem on the interval $[0,4\,T]$:
		\vspace*{-5pt}
	\begin{equation}\label{e-P-4T-2n}\tag{$P,\,4\,T$} \left\{\begin{array}{l} \widetilde{\widetilde{L}}\,u(t)=\tilde{\tilde \sigma}(t),\quad \mbox{ a. e.   }\; t\in [0,4\,T],\\[5pt]
			u^{(k)}(0)=u^{(k)}(4\,T),\ k=0,\dots,2n-1.
		\end{array}\right.\end{equation}
\end{itemize}

Now, we will show how to relate the expressions of different Green's functions. The following argument is analogous to the one made in \cite{CabCidSom2016} for the case $n=1$. 


\subsection{Neumann Problem}
To begin with, we will decompose the Green's function related to problem \eqref{e-N-2n} as sum of the Green's function related to \eqref{e-P-2T-2n} evaluated in the same point and of the same function evaluated in another point which satisfies a symmetric relation.

First, suppose that operator $L$ is nonresonant in the space
\[X_{N,T}=\left\{u\in W^{2n,1}(I): \   u^{(2k+1)}(0)=u^{(2k+1)}(T)=0, \ k=0,\dots,n-1\right\}. \]

Moreover, assume that $\widetilde L$ is nonresonant in 
\[X_{P,2T}=\left\{u\in W^{2n,1}(J): \   u^{(k)}(0)=u^{(k)}(2\,T), \ k=0,\dots,2n-1\right\}. \]



Then, if we denote by $G_N[T]$ and $G_P[2\,T]$ the Green's functions related to problems \eqref{e-N-2n} and \eqref{e-P-2T-2n}, respectively, reasoning as in \cite{CabCidSom2016}, we can deduce that
%
	\begin{equation}\label{e-Green-NP-2n}
		G_N[T](t,s)=G_P[2\,T](t,s)+G_P[2\,T](2\,T-t,s) \qquad \forall \; (t,s)\in I \times I.
	\end{equation}
	
	Previous expression lets us obtain the exact value at every point of the Green's function for the Neumann problem by means of the values of the periodic one, as long as both Green's functions exist. 
	
	Analogously, assuming that $\widetilde{L}$ is nonresonant in
	\[X_{N,2T}=\left\{u\in W^{2n,1}(J): \   u^{(2k+1)}(0)=u^{(2k+1)}(2T)=0, \ k=0,\dots,n-1\right\},\]
	and denoting by $G_N[2\,T]$ the Green's function related to \eqref{e-N-2T-2n}, we deduce that
	\begin{equation}\label{e-Green-NN-2n}
		G_N[T](t,s)=G_N[2\,T](t,s)+G_N[2\,T](2\,T-t,s) \qquad \forall \; (t,s)\in I \times I,
	\end{equation}
	or, using \eqref{e-Green-NP-2n},
	\begin{equation}\label{e-Green-NP-4T-2n}
		G_N[T](t,s)=G_P[4\,T](t,s)+G_P[4\,T](4\,T-t,s) +G_P[4\,T](2\,T-t,s)+G_P[4\,T](2\,T+t,s).
	\end{equation}
	
	\subsection{Dirichlet Problem}
	Now, we will do an analogous decomposition for the Green's function related to problem \eqref{e-D-2n}.
	
	To this end, we will assume that operator $L$ is nonresonant in 
	\[X_{D,T}=\left\{u\in W^{2n,1}(I): \   u^{(2k)}(0)=u^{(2k)}(T)=0, \ k=0,\dots,n-1\right\}. \]
	
	Again, we will also assume that $\widetilde L$ is nonresonant in $X_{P,2T}$.
	
	Now, denoting by $G_D[T]$ the Green's function related to \eqref{e-D-2n}, we obtain:
	\begin{equation}\label{e-Green-DP-2n}
		G_D[T](t,s)=G_P[2\,T](t,s)-G_P[2\,T](2\,T-t,s) \qquad \forall \; (t,s)\in I \times I.
	\end{equation}
	
	On the other hand, assuming that $\widetilde{L}$ is nonresonant in 
	\[X_{D,2T}=\left\{u\in W^{2n,1}(J): \   u^{(2k)}(0)=u^{(2k)}(2T)=0, \ k=0,\dots,n-1\right\}, \]
and denoting by $G_D[2\,T]$ the Green's function related to \eqref{e-D-2T-2n}, we obtain that
	\begin{equation}\label{e-Green-DD-2n}
		G_D[T](t,s)=G_D[2\,T](t,s)-G_D[2\,T](2\,T-t,s) \qquad \forall \; (t,s)\in I \times I
	\end{equation}
	or, using \eqref{e-Green-DP-2n},
	\begin{equation}\label{e-Green-DP-4T-2n}
		G_D[T](t,s)=G_P[4\,T](t,s)-G_P[4\,T](4\,T-t,s) -G_P[4\,T](2\,T-t,s)+G_P[4\,T](2\,T+t,s). 
	\end{equation}
	
	\subsection{Mixed Problems}
	The same arguments of two previous sections are applicable to problems \eqref{e-M1-2n} and \eqref{e-M2-2n}, by assuming the nonresonant character of operator $L$ in 
	
	\vspace*{-15pt}
	\[X_{M_1,T}=\left\{u\in W^{2n,1}(I): \   u^{(2k+1)}(0)=u^{(2k)}(T)=0, \ k=0,\dots,n-1\right\}\]
	\vspace{-8pt}
	or
	\[X_{M_2,T}=\left\{u\in W^{2n,1}(I): \   u^{(2k)}(0)=u^{(2k+1)}(T)=0, \ k=0,\dots,n-1\right\},\]
	respectively. However, these problems will not be related to periodic ones but to \eqref{e-A-2T-2n}. Therefore, we will assume for both cases that operator $\widetilde{L}$ is nonresonant in \[X_{A,2T}=\left\{u\in W^{2n,1}(J): \   u^{(k)}(0)=-u^{(k)}(2T), \ k=0,\dots,2\,n-1\right\}.\]
	Then  we arrive at the following decompositions:
	\begin{equation}\label{e-Green-M1A-2n}
		G_{M_1}[T](t,s)=G_A[2\,T](t,s)-G_A[2\,T](2\,T-t,s) \qquad \forall \; (t,s)\in I \times I,
	\end{equation}
	\begin{equation}\label{e-Green-M2A-2n}
		G_{M_2}[T](t,s)=G_A[2\,T](t,s)+G_A[2\,T](2\,T-t,s) \qquad \forall \; (t,s)\in I \times I.
	\end{equation}
	Again, $G_{M_1}[T]$ and $G_{M_2}[T]$ can also be related to $G_N[2T]$ and $G_D[2T]$, respectively:
	\begin{equation}\label{e-Green-M1N-2n}
		G_{M_1}[T](t,s)=G_N[2\,T](t,s)-G_N[2\,T](2\,T-t,s) \qquad \forall \; (t,s)\in I \times I,
	\end{equation}
	\begin{equation}\label{e-Green-M2D-2n}
		G_{M_2}[T](t,s)=G_D[2\,T](t,s)+G_D[2\,T](2\,T-t,s) \qquad \forall \; (t,s)\in I \times I,
	\end{equation}
	or, using \eqref{e-Green-NP-2n} and \eqref{e-Green-DP-2n},
	\begin{equation}\label{e-Green-M1P-4T-2n}
		G_{M_1}[T](t,s)=G_P[4\,T](t,s)+G_P[4\,T](4\,T-t,s) -G_P[4\,T](2\,T-t,s)-G_P[4\,T](2\,T+t,s), 
	\end{equation}
	\begin{equation}\label{e-Green-M2P-4T-2n}
		G_{M_2}[T](t,s)=G_P[4\,T](t,s)-G_P[4\,T](4\,T-t,s) +G_P[4\,T](2\,T-t,s)-G_P[4\,T](2\,T+t,s).
	\end{equation}
	
	\vspace*{2pt}
	On the other hand, it is also possible to obtain a direct relation between the Green's functions of the two mixed problems. 
	
	Consider the following operator defined from $L$ by taking the reflection of the coefficients
	\vspace*{-8pt}
	\begin{equation*}\begin{split}
			\check{L}\,u(t) & 
			= u^{(2n)}(t)+\sum_{k=0}^{2n-1}(-1)^k\,a_k(T-t)\,u^{(k)}(t),
				\vspace*{-20pt}
		\end{split}\end{equation*}
		for all $t\in I$, and let $\check{G}_{M_2}[T]$ be the Green's function related to the Mixed problem 2 associated with $\check{L}$, namely,
		\begin{equation}\label{e-M2-checkL}\left\{\begin{split}
				& \check{L}\,u(t)=\check{\sigma}(t), \quad t\in I,\\[2pt]
				& u^{(2k)}(0)=u^{(2k+1)}(T)=0, \quad k=0,\dots,n-1.
			\end{split}\right.\end{equation}

		Now, let $u$ be the unique solution of problem \eqref{e-M1-2n}.		
		If we define $v(t)=u(T-t)$, it is easy to check that $v$ is a solution of problem \eqref{e-M2-checkL} for the particular case of taking $\check{\sigma}(t)=\sigma(T-t)$. Therefore,
		\[v(t)=\int_{0}^{T} \check{G}_{M_2}[T](t,s)\,\sigma(T-s)\, \dif s= \int_{0}^{T} \check{G}_{M_2}[T](t,T-s)\,\sigma(s)\, \dif s.\]
		
		On the other hand, 
				\vspace*{-5pt}
		\[v(t)=u(T-t)=\int_{0}^{T} G_{M_1}[T](T-t,s)\,\sigma(s)\, \dif s. \]
		
		\vspace*{-2pt}
		Since previous equalities are valid for all $\sigma \in \Lsp{1}(I)$, we deduce that
		\[G_{M_1}[T](T-t,s)=\check{G}_{M_2}[T](t,T-s) \]
		or, which is the same,
		\vspace*{-8pt}
		\begin{equation}\label{e-M1M2-1-2n}
			G_{M_1}[T](T-t,T-s)=\check{G}_{M_2}[T](t,s).
		\end{equation}
		
		\vspace*{-2pt}
		Analogously, if we denote by $\check{G}_{M_1}[T]$ the Green's function related to the Mixed problem 1 associated with $\check{L}$, namely,
		\vspace*{-5pt}
		\begin{equation}\label{e-M1-checkL}
			\left\{\begin{split}
				& \check{L}\,u(t)=\check{\sigma}(t), \quad t\in I,\\
				& u^{(2k+1)}(0)=u^{(2k)}(T)=0, \quad k=0,\dots,n-1,
			\end{split}\right.\end{equation} 
			
		\vspace*{-5pt}
		\noindent repeating the previous reasoning, we reach to the following connecting expression
		\begin{equation}\label{e-M1M2-2-2n}
			G_{M_2}[T](T-t,T-s)=\check{G}_{M_1}[T](t,s).
		\end{equation}
		
		\subsection{Connecting relations between different problems}
		On the other hand, assuming again the nonresonant character of all the considered operators in the corresponding spaces, if we sum different combinations of the previous equalities, we obtain more connecting expressions between the considered Green's functions. The results are the following:
		\begin{itemize}
			\item From \eqref{e-Green-NP-2n} and \eqref{e-Green-DP-2n}, we deduce:
			\begin{equation}\label{e-Green-DN-P-2n}\begin{split}
					G_P[2\,T](t,s)&=1/2\left(G_N[T](t,s)+G_D[T](t,s)\right) \qquad \forall \; (t,s)\in I \times I, \\
					G_P[2\,T](2\,T-t,s)&=1/2 \left(G_N[T](t,s)-G_D[T](t,s)\right) \qquad \forall \; (t,s)\in I \times I.
				\end{split}\end{equation}
				
				\item From \eqref{e-Green-M1A-2n} and \eqref{e-Green-M2A-2n}:
				\begin{equation}\label{e-Green-M1M2-A-2n}\begin{split}
						G_A[2\,T](t,s)&=1/2 \left(G_{M_2}[T](t,s)+G_{M_1}[T](t,s)\right) \qquad \forall \; (t,s)\in I \times I,\\
						G_A[2\,T](2\,T-t,s)&=1/2 \left(G_{M_2}[T](t,s)-G_{M_1}[T](t,s)\right) \qquad \forall \; (t,s)\in I \times I.
					\end{split}\end{equation}
					
					\vspace*{-5pt}
					\item From \eqref{e-Green-NN-2n} and \eqref{e-Green-M1N-2n}:
					\vspace*{-3pt}
					\begin{equation}\label{e-Green-NM1-N-2n}\begin{split}
							G_N[2\,T](t,s)&=1/2\left(G_N[T](t,s) +G_{M_1}[T](t,s)\right) \qquad \forall \; (t,s)\in I \times I, \\
							G_N[2\,T](2\,T-t,s)&=1/2\left(G_N[T](t,s)-G_{M_1}[T](t,s)\right) \qquad \forall \; (t,s)\in I \times I.
						\end{split}\end{equation}
						
						\vspace*{-5pt}
						\item From \eqref{e-Green-DD-2n} and \eqref{e-Green-M2D-2n}:
						\vspace*{-3pt}
						\begin{equation}\label{e-Green-DM2-D-2n}\begin{split}
								G_D[2\,T](t,s)&=1/2\left(G_{M_2}[T](t,s) +G_D[T](t,s)\right) \qquad \forall \; (t,s)\in I \times I, \\
								G_D[2\,T](2\,T-t,s)&=1/2\left(G_{M_2}[T](t,s)-G_{D}[T](t,s)\right) \qquad \forall \; (t,s)\in I \times I.
							\end{split}\end{equation}
							
							\vspace*{-5pt}
							\item From \eqref{e-Green-NP-4T-2n}, \eqref{e-Green-DP-4T-2n}, \eqref{e-Green-M1P-4T-2n} and \eqref{e-Green-M2P-4T-2n}:
							\vspace*{-3pt}
							\begin{equation*}
								G_P[4\,T](t,s)=1/4\left(G_N[T](t,s) + G_D[T](t,s)+ G_{M_1}[T](t,s) +G_{M_2}[T](t,s)\right).
							\end{equation*}	
						\end{itemize}

\section{Decomposition of the Spectra}\label{S-decomp-spectra}
In this section we will show first how the spectra of the considered problems can be connected. The results here generalise those proved in \cite{CabCidSom2016}.
						
We will denote by $\Lambda_N[T]$, $\Lambda_D[T]$, $\Lambda_{M_1}[T]$, $\Lambda_{M_2}[T]$, $\Lambda_P[2\,T]$, $\Lambda_A[2\,T]$, $\Lambda_N[2\,T]$, $\Lambda_D[2\,T]$ and $\Lambda_P[4\,T]$ the set of eigenvalues of problems \eqref{e-N-2n}, \eqref{e-D-2n}, \eqref{e-M1-2n}, \eqref{e-M2-2n}, \eqref{e-P-2T-2n}, \eqref{e-A-2T-2n}, \eqref{e-N-2T-2n}, \eqref{e-D-2T-2n} and \eqref{e-P-4T-2n}, respectively. Then, arguing as in \cite{CabCidSom2016}, we obtain the following equalities:			
\[\Lambda_N[T] \cup \Lambda_D[T] = \Lambda_P[2\,T],\]
\[\Lambda_N[T]\cup \Lambda_{M_1}[T] =\Lambda_N[2T],\]	
\[\Lambda_D[T]\cup \Lambda_{M_2}[T]=\Lambda_D[2T],\]	
\[\Lambda_{M_1}[T]\cup \Lambda_{M_2}[T]=\Lambda_A[2T],\]	
\[\Lambda_N[T]\cup \Lambda_D[T] \cup \Lambda_{M_1}[T] \cup \Lambda_{M_2}[T]= \Lambda_P[4T].\]

						Finally, if we denote by $\check{\Lambda}_{M_2}[T]$ and $\check{\Lambda}_{M_1}[T]$ the set of eigenvalues of problems \eqref{e-M2-checkL} and \eqref{e-M1-checkL}, respectively, from \eqref{e-M1M2-1-2n} and \eqref{e-M1M2-2-2n} we deduce that
						\[\Lambda_{M_1}[T]=\check{\Lambda}_{M_2}[T] \quad \text{and} \quad \Lambda_{M_2}[T]=\check{\Lambda}_{M_1}[T]. \]

						As an immediate consequence we have the following result.
						\begin{corollary}
							If $a_k(t)=(-1)^k a_k(T-t)$ for all $k=0,\dots,2n-1$, then the spectra of the two mixed problems coincides, that is,
							\[\Lambda_{M_1}[T]=\Lambda_{M_2}[T]. \]
						\end{corollary}
						
						Moreover, if we denote by $\lambda_0^N[T]$, $\lambda_0^D[T]$, $\lambda_0^{M_1}[T]$, $\lambda_0^{M_2}[T]$, $\lambda_0^P[2\,T]$, $\lambda_0^A[2\,T]$, $\lambda_0^N[2\,T]$, $\lambda_0^D[2\,T]$ and $\lambda_0^P[4\,T]$ the first eigenvalue of problems \eqref{e-N-2n}, \eqref{e-D-2n}, \eqref{e-M1-2n}, \eqref{e-M2-2n}, \eqref{e-P-2T-2n}, \eqref{e-A-2T-2n}, \eqref{e-N-2T-2n}, \eqref{e-D-2T-2n} and \eqref{e-P-4T-2n}, respectively, from the connecting expressions proved in Section \ref{S-Decomp-Green}, we deduce the relations below.
						
						\begin{theorem}\label{t-order-eigen-2n} Assume that all the previously considered spectra are not empty, the first eigenvalue of each problem is simple and its related eigenfunction has constant sign. Then, the following equalities are fulfilled for any $a_0,\dots,a_{2n-1} \in \Lsp{1}(I)$:
							\begin{enumerate}
								\item $\lambda_0^N[T]=\lambda_0^P[2\,T]< \lambda_0^D[T]$.
								\item $\lambda_0^N[T]= \lambda_0^N[2\,T]< \lambda_0^{M_1}[T]$.
								\item $\lambda_0^N[T]=\lambda_0^P[4\,T]$.
								\item $\lambda_0^{M_2}[T]=\lambda_0^D[2\,T]< \lambda_0^D[T]$.
								\item $\lambda_0^N[T]< \lambda_0^{M_2}[T]$.
								\item $\lambda_0^A[2\,T]=\min\left\{\lambda_0^{M_1}[T], \, \lambda_0^{M_2}[T] \right\}$.
								
							\end{enumerate}
						\end{theorem} 
						
						\begin{proof}
							Assertion $1$ is proved in the following way: as we have seen above, the spectrum of \eqref{e-P-2T-2n} is decomposed as $\Lambda_P[2\,T]=\Lambda_N[T] \cup \Lambda_D[T]$, which implies that
							\[\lambda_0^P[2\,T]=\min\left\{\lambda_0^N[T], \, \lambda_0^D[T] \right\} . \]
Consider now the even extension to $J$ of the eigenfunction associated to $\lambda_0^N[T]$. This extension has constant sign on $J$ and, moreover, it satisfies periodic boundary conditions, so it is a constant sign eigenfunction of \eqref{e-P-2T-2n}. On the contrary, the odd extension to $J$ of the eigenfunction associated to $\lambda_0^D[T]$ is a sign changing eigenfunction of \eqref{e-P-2T-2n}. Thus, since the eigenfunction related to the first eigenvalue of each problem has constant sign, we deduce that $\lambda_0^N[T]=\lambda_0^P[2\,T]<\lambda_0^D[T]$. 
							
							An analogous argument is valid to prove Assertion $2$, by taking into account that $\Lambda_N[2\,T]=\Lambda_N[T] \cup \Lambda_{M_1}[T]$.
							
							Assertion $3$ is deduced from the two previous one. Indeed, Assertion $1$ implies that $\lambda_0^N[2\,T]=\lambda_0^P[4\,T]$ and, from Assertion $2$, we deduce the equality.
							
							Assertion $4$ is proved analogously to Assertions $1$ and $2$, taking into account the decomposition $\Lambda_D[2\,T]=\Lambda_D[T] \cup \Lambda_{M_2}[T]$.
							
							Now Assertion $5$ can be deduced from $1$, $2$ and $4$. Indeed, Assertion $1$ implies that $\lambda_0^N[2\,T]< \lambda_0^D[2\,T]$ and, using Assertions $2$ and $4$, $\lambda_0^N[T]=\lambda_0^N[2\,T]< \lambda_0^D[2\,T]=\lambda_0^{M_2}[T]$.
							
							Finally, Assertion $6$ is immediate from  $\Lambda_A[2\,T]=\Lambda_{M_1}[T]\cup \Lambda_{M_2}[T].$
						\end{proof}
						
						\begin{remark}
							With respect to the hypothesis that all the considered spectra are not empty note that, as a consequence of the relations proved at the beginning of this section, if one of those spectra is not empty, we could ensure that some others are not empty too.
							
							On the other hand, there are several results which ensure that, under some suitable conditions, the first eigenvalue of a boundary value problem is simple and its related eigenfunction has constant sign, for instance, Krein-Rutman Theorem.
							
							Sufficient conditions to ensure that all the hypotheses required in previous theorem are fulfilled can be found in \cite{Karlin64}. First, we can deduce from Theorem~1 in such reference that if there exists some $\lambda$ for which the Green's function $G[\lambda,T]$ has constant sign and the spectrum of such problem is not empty, then the eigenfunction related to the first eigenvalue has constant sign. 
							
							Moreover, from Theorem~2 in \cite{Karlin64} it is deduced that if there exists some $\lambda$ for which the Green's function $G[\lambda,T]$ has strict constant sign on $[0,T]\times (0,T)$ then the spectrum of such problem is not empty, the first eigenvalue is simple and its related eigenfunction has strict constant sign. 
							
							Finally, from Theorem~2' in \cite{Karlin64} we can ensure that if there exists some $\lambda$ for which $G[\lambda,T]$ has strict constant sign on $(0,T)\times (0,T)$ and there exists a continuous function $\phi$, positive on $(0,T)$, such that \[\frac{G[\lambda,T](t,s)}{\phi(t)}\] is continuous on $[0,T]\times [0,T]$ and positive on $[0,T]\times (0,T)$, then the spectrum of such problem is not empty, the first eigenvalue is simple and its related eigenfunction has strict constant sign. 
							
							Analogously, if conditions given in Lemmas~\ref{l-M-NT-2n} or \ref{l-M-PT-2n} hold for some $\lambda$, then we are also able to deduce that the spectrum of such problem is not empty, the first eigenvalue is simple and its related eigenfunction has constant sign. Details of this can be seen in \cite{Cablibro}, where it is proved that Lemmas~\ref{l-M-NT-2n} or \ref{l-M-PT-2n} imply that Krein-Rutman's Theorem holds.
							
							Finally, we must note that, since the eigenfunctions of the considered problems are related, the constant sign of the eigenfunction associated with the first eigenvalue of a problem implies (in some cases) the constant sign of the eigenfunction of other problems.
						\end{remark}
						
\section{Constant sign of Green's functions}
In \cite{CabCidSom2016} and \cite[Section 3.4]{CabCidSomLibro}, for $n=1$, some results relating the constant sign of various Green's functions have been proved. The result is the following:
			\begin{theorem}\cite[Corollary 4.8]{CabCidSom2016} \label{t-G-signo-cte-2n}
				For $n=1$ and $a_1\equiv 0$, the following properties hold:
		\begin{enumerate}
		\item $G_P[2\,T]<0$ on $J \times J$ if and only if $G_N[T]<0$ on $I \times I$. This is equivalent to $G_N[2\,T]<0$ on $J\times J$.
								\item $G_P[2\,T]> 0$ on $(0,2\,T) \times (0,2\,T)$ if and only if $G_N[T]> 0$ on $(0,T) \times (0,T)$.
						\item If $G_N[2\,T]> 0$ on $(0,2\,T) \times (0,2\,T)$ then $G_N[T]>0$ on $(0,T) \times (0,T)$.
								\item If $G_P[2\,T]<0$ on $J \times J$ then $G_D[2\,T]<0$ on $(0,2\,T) \times (0,2\,T)$.
				\item If $G_P[2\,T]> 0$ on $(0,2\,T) \times (0,2\,T)$ then $G_D[2\,T]<0$ on $(0,2\,T) \times (0,2\,T)$.
								\item If $G_N[T]$ (or $G_P[2\,T]$) has constant sign on $I\times I$, then $G_D[T]<0$ on $(0,T) \times (0,T)$, $G_{M_1}[T]<0$ on $[0,T)\times[0,T)$ and $G_{M_2}[T]<0$ on $(0,T] \times (0,T]$.
								\item $G_D[2\,T]<0$ on $(0,2\,T) \times (0,2\,T)$ if and only if $G_{M_2}[T]< 0$ on $(0,T]\times (0,T]$.
								\item If $G_{M_2}[T]<0$ on $(0,T]\times (0,T]$ or $G_{M_1}[T]<0$ on $[0,T) \times [0,T)$ then $G_D[T]<0$ on $(0,T) \times (0,T)$.
							\end{enumerate}
						\end{theorem}
						
Some of the previous results can be extended to the more general case considered in this paper, although some of them are no longer true. We will show now these results and some counterexamples for the cases which do not hold anymore.
						
						From \eqref{e-Green-NP-2n}, \eqref{e-Green-NN-2n} and \eqref{e-Green-M2D-2n} we deduce that
						\begin{corollary}\label{c-relac-sign-Green-2n} The following properties are fulfilled for any $a_0,\dots,a_{2n-1} \in \Lsp{1}(I)$:
							\begin{enumerate}
								\item If $G_P[2T]\le 0$ on $J\times J$, then $G_N[T]\le 0$ on $I\times I$.
								\item If $G_P[2T]\ge 0$ on $J\times J$, then $G_N[T]\ge 0$ on $I\times I$.
								\item If $G_N[2T]\le 0$ on $J\times J$, then $G_N[T]\le 0$ on $I\times I$.
								\item If $G_N[2T]\ge 0$ on $J\times J$, then $G_N[T]\ge 0$ on $I\times I$.
								\item If $G_D[2T]\le 0$ on $J\times J$, then $G_{M_2}[T]\le 0$ on $I\times I$.
								\item If $G_D[2T]\ge 0$ on $J\times J$, then $G_{M_2}[T]\ge 0$ on $I\times I$.
							\end{enumerate}
						\end{corollary}

	The reciprocal of Assertions 1 and 2 in Corollary~\ref{c-relac-sign-Green-2n} holds for constant coefficients. This occurs as a consequence of the following property:
					\begin{lemma}\cite[Section 1.4]{Cablibro} \label{l-coeff-const-2n}
							Let $L_n\,u(t)\equiv u^{(n)}(t)+a_{n-1}(t)\,u_{n-1}(t) +\dots + a_1(t)\, u'(t)+a_0(t)\,u(t)$, $t\in I $
							be a $n$-th order linear operator and let $G_P[T]$ denote the Green's function related to the periodic problem
							\begin{equation*}\left\{\begin{split}
									& L_n\,u(t)=0, \quad t\in I,\\
									& u^{(k)}(0)=u^{(k)}(T), \quad k=0,\dots,n-1.
								\end{split}\right.\end{equation*}
							
							If the coefficients $a_k$, $k=0,\dots,n-1$, involved in the definition of operator $L_n$ are constant, then the Green's function is constant over the straight lines of slope one, that is, it satisfies the following property
							\begin{equation*}G_P[T](t,s)=\left\{\begin{array}{ll}
									G_P[T](t-s,0), & 0\le s\le t\le T, \\[4pt]
									G_P[T](T+t-s,0), & \text{otherwise.}
								\end{array}\right.
							\end{equation*}
						\end{lemma}
						
						As a consequence, we arrive at the following result.
						\begin{theorem}
							If all the coefficients $a_0,\dots,a_{2n-1}$ of operator $L$ defined in \eqref{e-L-2n} are constant, then the following properties hold:
							\begin{enumerate}
								\item $G_P[2T]\le 0$ on $J\times J$ if and only if $G_N[T]\le 0$ on $I\times I$.
								\item $G_P[2T]\ge 0$ on $J\times J$ if and only if $G_N[T]\ge 0$ on $I\times I$.
							\end{enumerate}
						\end{theorem}
						
						\begin{proof}
							From Corollary~\ref{c-relac-sign-Green-2n}, the Assertion is equivalent to prove that if $G_P[2\,T]$ changes sign, then $G_N[T]$ will also change sign. Indeed, assume that there exist two pairs of values $(t_1,s_1)$ and $(t_2,s_2)$ such that $G_P[2\,T](t_1,s_1)<0$ and $G_P[2\,T](t_2,s_2)>0$. As $G_P[2\,T](t,s)=G_P[2\,T](s,t)$ for all $(t,s)\in J\times J$, we may assume, without loss of generality, that $s_1\le t_1$ and $s_2\le t_2$.
							
							If all the coefficients $a_0,\dots,a_{2n-1}$ are constant then, from Lemma \ref{l-coeff-const-2n}, it holds that 
							\begin{equation*}G_P[2\,T](t,s)=\left\{\begin{array}{ll}
									G_P[2\,T](t-s,0), & 0\le s\le t\le 2\,T, \\[4pt]
									G_P[2\,T](2\,T+t-s,0), & \text{otherwise.}
								\end{array}\right.
							\end{equation*}
							
							Therefore, it is fulfilled that
							\[G_P[2\,T](t_1,s_1) =G_P[2\,T](t_1-s_1,0) \quad \text{and} \quad G_P[2\,T](t_2,s_2) =G_P[2\,T](t_2-s_2,0).\]
							
							On the other hand, from equality \eqref{e-l-sim-G-2n} and the fact that the Green's function satisfies the periodic boundary conditions (see \cite[Definition 1.4.1]{Cablibro}), it holds that 
							\[G_P[2\,T](t_1-s_1,0)=G_P[2\,T](2\,T-t_1+s_1,2\,T)= G_P[2\,T](2\,T-t_1+s_1,0),\]
							\[G_P[2\,T](t_2-s_2,0)=G_P[2\,T](2\,T-t_2+s_2,2\,T)=G_P[2\,T](2\,T-t_2+s_2,0).\]
							
							Now, we will distinguish two possibilities:
							\begin{itemize}
								\item If $t_1-s_1\le T$, then
								\begin{equation*}\begin{split}
										G_N[T](t_1-s_1,0)&=G_P[2\,T](t_1-s_1,0)+G_P[2\,T](2\,T-t_1+s_1,0)\\
										&=2\,G_P[2\,T](t_1-s_1,0)<0.
									\end{split}\end{equation*}
									\item When $t_1-s_1>T$, we have
									\begin{equation*}\begin{split}
											G_N[T](2\,T-t_1+s_1,0)&=G_P[2\,T](2\,T-t_1+s_1,0)+G_P[2\,T](t_1-s_1,0)\\
											&=2\,G_P[2\,T](t_1-s_1,0)<0.
										\end{split}\end{equation*}
									\end{itemize}
									Analogously, if $t_2-s_2\le T$, then 
									\begin{equation*}\begin{split}
											G_N[T](t_2-s_2,0)&=2\,G_P[2\,T](t_2-s_2,0)>0,
										\end{split}\end{equation*}
										and, if $t_2-s_2> T$, then 
										\begin{equation*}\begin{split}
												G_N[T](2\,T-t_2+s_2,0)&=2\,G_P[2\,T](t_2-s_2,0)>0.
											\end{split}\end{equation*}
											It is clear that, in any of the cases, $G_N[T]$ changes its sign and the result holds.
										\end{proof}

										The following counterexample shows that the converse of Assertion 2 in Corollary \ref{c-relac-sign-Green-2n} is not true in general for nonconstant coefficients.
										\begin{example}{Example}
											Consider the Neumann problem on $[0,T]=\left[0,\,2\right]$ related to operator 
											\begin{equation}\label{ex-N-NP-2n}\begin{split}
													& L\,u(t)=u^{(4)}(t)+((t-2)^4+\lambda)\,u(t), \quad t\in\left[0,\,2\right],
												\end{split}\end{equation}
												and the periodic problem on $[0,2\,T]=[0,4]$ related to 
												\begin{equation}\label{ex-P-NP-2n}\begin{split}
														& \widetilde{L}\,u(t)\equiv u^{(4)}(t)+((t-2)^4+\lambda)\,u(t), \quad t\in\left[0,\,4\right].
													\end{split}\end{equation}
													
													By numerical approach, it can be seen that $G_N[T]$ is nonpositive for $\lambda\in \left(\lambda_1,\,\lambda_0^N[T] \right)$, where $\lambda_1\approx -2.26$ and $\lambda_0^N[T]=\lambda_0^P[2\,T]\approx -1.746$. Moreover it is nonnegative for $\lambda\in \left( \lambda_0^N[T],\, \lambda_2 \right)$, with $\lambda_2\approx 4.11$. However, $G_P[2\,T]$ is nonpositive for $\lambda\in \left(\lambda_1,\,\lambda_0^P[2\,T] \right)$ and nonnegative for $\lambda\in \left( \lambda_0^P[2\,T],\, \lambda_3\right)$, with  $\lambda_3\approx 5.95$. 
													
													Despite this, we remark that the interval of values of $\lambda$ for which $G_N[T]$ and $G_P[2\,T]$ are nonpositive is exactly the same.
												\end{example}
												
												\begin{remark}
													It must be pointed out that the converse of Assertion 2 in Corollary \ref{c-relac-sign-Green-2n} also holds for several examples with non constant coefficients. However we have not been able to prove the existence of any general condition under which this Assertion holds.
													
													Furthermore, up to this moment, we have not been able to find a counterexample for the converse of Assertion 1. So, it remains as an open problem to know if Assertion 1 is or not an equivalence for $n\ge 2$.
												\end{remark}
												
												The following counterexample shows that the converse of Assertions 3 and 4 in Corollary \ref{c-relac-sign-Green-2n} does not hold in general, not even in the constant case: 
												
												\begin{example}{Example}
													Consider the following Neumann problem with constant coefficients on $[0,T]=\left[0,\,\frac{3}{2}\right]$ related to the following operator
													\begin{equation*}
														\begin{split}
															L\,u(t)\equiv u^{(4)}(t)+\lambda\,u(t), \quad t\in\left[0,\,\frac{3}{2}\right],
														\end{split}\end{equation*}
														and the Neumann problem  on $[0,2\,T]=[0,3]$ related to
														\begin{equation*}
															\begin{split}
																\widetilde{L}\,u(t)\equiv u^{(4)}(t)+\lambda\,u(t), \quad t\in\left[0,\,3\right],
															\end{split}\end{equation*}
															
															By numerical approach, it can be seen that $G_N[T]$ is nonpositive for $\lambda\in \left(\lambda_4,\,\lambda_0^N[T] \right)$, with $\lambda_4\approx -6.1798$ and $\lambda_0^N[T]=\lambda_0^N[2\,T]=0$, and nonnegative for $\lambda\in \left( \lambda_0^N[T],\, \lambda_5 \right)$, with $\lambda_5\approx 24.7192$. However, $G_N[2\,T]$ is nonpositive for $\lambda\in \left(\lambda_6,\,\lambda_0^N[2\,T] \right)$, with $\lambda_6\approx -0.3862$, and nonnegative for $\lambda\in \left( \lambda_0^N[2\,T],\, \lambda_7 \right)$, with $\lambda_7\approx 1.5449$.
															
															So, the converse of Assertions 3 and 4 does not hold for these operators.
														\end{example}

														%

														The following counterexample shows that the converse of Assertions 5 and 6 in Corollary \ref{c-relac-sign-Green-2n} is not true in general, not even in the constant case:
														\begin{example}{Example}
															Consider the Mixed problem 2 with constant coefficients on $[0,T]=\left[0,\,1\right]$ related to operator 
															\begin{equation*}
																\begin{split}
																	& L\,u(t)\equiv u^{(4)}(t)+\lambda\,u(t), \quad t\in\left[0,\,1\right],
																\end{split}\end{equation*}
																and the Dirichlet problem  on $[0,2\,T]=[0,2]$ related to
																\begin{equation*}
																	\begin{split}
																		& \widetilde{L}\,u(t)\equiv u^{(4)}(t)+\lambda\,u(t), \quad t\in\left[0,\,2\right].
																	\end{split}\end{equation*}
																	
																	It can be seen that $G_{M_2}[T]$ is nonpositive for $\lambda\in \left(\lambda_8,\,\lambda_0^{M_2}[T] \right)$, with $\lambda_8\approx -31.2852 $ and $\lambda_0^{M_2}[T] =\lambda_0^D[2\,T]=-\frac{\pi^4}{16}$, and nonnegative for $\lambda\in \left( \lambda_0^{M_2}[T],\, \lambda_9\right)$, with $\lambda_9\approx 389.6365$. However, $G_D[2\,T]$ is nonpositive for $\lambda\in \left(\lambda_{10},\,\lambda_0^D[2\,T] \right)$, where $\lambda_{10}\approx -14.8576$, and nonnegative for $\lambda\in \left(\lambda_0^D[2\,T],\,\lambda_{11}\right)$, with $\lambda_{11}\approx 59.4303$.
																\end{example}
																
																%
																
																We will see now some more counterexamples which show that Assertions 4, 5, 6 and 8 in Theorem~\ref{t-G-signo-cte-2n} do not hold, in general, for $n>1$.
																
																Next example shows that Assertions 4 and 5 in Theorem~\ref{t-G-signo-cte-2n} are not true in general.
																
																\begin{example}{Example}
																	Consider the periodic and Dirichlet problems on the same interval $[0,2\,T]=[0,3]$ related to operator
																	\begin{equation}\label{ex-P-PD-2n}
																		\begin{split}
																			\widetilde{L}\,u(t) \equiv u^{(4)}(t)+(t\,(t-3)+\lambda)\,u(t), \quad t\in [0,3].
																		\end{split} \end{equation}
																		
																		By numerical approach, we have obtained that for $\lambda=-1.5$, $G_P[2\,T]$ is negative while $G_D[2\,T]$ changes its sign on $J\times J$. 
																		
																		Moreover, for $\lambda=15$, $G_P[2\,T]$ is positive while $G_D[2\,T]$ changes sign again. 
																	\end{example}

																	We will see in the two following examples that none of the implications given in Assertion 6 in Theorem~\ref{t-G-signo-cte-2n} holds for $n>1$.
																	
																	\begin{example}{Example}
																		Consider now $[0,T]=[0,2]$ and operators $L$ and $\widetilde{L}$ given in \eqref{ex-N-NP-2n} and \eqref{ex-P-NP-2n}.
																		
																		For $\lambda=-2$, one can check that both $G_P[2\,T]$ and $G_N[T]$ are nonpositive, meanwhile $G_D[T]$ and $G_{M_1}[T]$ are nonnegative. 
																		
																		For $\lambda=2$, it occurs that  both $G_P[2\,T]$ and $G_N[T]$ are nonnegative, meanwhile $G_D[T]$, $G_{M_1}[T]$ and $G_{M_2}[T]$ are nonnegative. 
																	\end{example}
																	
																	\begin{example}{Example}
																		Take now $[0,T]=\left[0,\frac{3}{2}\right]$, operator $L$ given by
																		\begin{equation*}
																			L\,u (t)\equiv u^{(4)}(t)+\left(t(t-3)+\lambda\right), \quad t\in \left[0,\frac{3}{2}\right]
																		\end{equation*}	
																		and operator $\widetilde{L}$ given in \eqref{ex-P-PD-2n}.
																		
																		In this case, for $\lambda=1.5$, it occurs that $G_P[2\,T]$ and $G_N[T]$ are nonpositive, meanwhile $G_{M_2}[T]$ is nonnegative.
																	\end{example}

																	Finally, we will show that Assertion 8 in Theorem~\ref{t-G-signo-cte-2n} does not hold either when $n>1$.
																	
																	\begin{example}{Example}
																		Consider again $[0,T]=[0,2]$  and operators $L$ and $\widetilde{L}$ given in \eqref{ex-N-NP-2n} and \eqref{ex-P-NP-2n}.
																		
																		In this case, for $\lambda=-6$, $G_{M_1}[T]$ is nonpositive but $G_D[T]$ is nonnegative.
																		
																		Similarly, for $\lambda=-2$, $G_{M_2}[T]$ is nonpositive but $G_D[T]$ is nonnegative.
																	\end{example}
																	
																	Finally, from the relations given in Theorem~\ref{t-order-eigen-2n}, together with the general characterization given in Lemmas \ref{l-M-NT-2n} and \ref{l-M-PT-2n}, it can be deduced the following corollary.
																	
																	\begin{corollary}
																		Assume that we are in conditions to apply Lemmas~\ref{l-M-NT-2n} and \ref{l-M-PT-2n}, that is, all the considered Green's functions $G[\lambda,T]$ (or $G[\lambda,2\,T]$, $G[\lambda,4\,T]$, with the suitable subscript for each case) are:
																		\begin{itemize}
																			\item nonpositive on $I \times I$ if and only if $\lambda\in (-\infty, \lambda_1)$ or  $\lambda\in [-\bar{\mu}, \lambda_1)$, with $\lambda_1 >0$ the first eigenvalue of operator $L_n$ coupled with the corresponding boundary conditions and ${\bar{\mu} \ge 0}$ such that $L_n[-\bar{\mu}]$ is nonresonant on $X$ and the related nonpositive Green's function $G[-\bar{\mu}]$ vanishes at some point of the square $I\times I$.
																			\item nonnegative on $I \times I$ if and only if $\lambda\in (\lambda_1,\infty)$ or  $\lambda\in (\lambda_1, \bar{\mu}]$, with $\lambda_1 <0$ the first eigenvalue of operator $L_n$ coupled with the corresponding boundary conditions and $\bar{\mu} \ge 0$ such that $L_n[\bar{\mu}]$ is nonresonant on $X$ and the related nonnegative Green's function $G[\bar{\mu}]$ vanishes at some point of the square $I\times I$.
																		\end{itemize} 
																		
																		Then the following relations between the constant sign of Green's functions are valid for any $a_0,\dots,a_{2n-1}\in \Lsp{1}(I)$:
																		\begin{itemize}
																			\item If $G_N[T]$ is nonpositive on $I\times I$, then $G_D[T]$, $G_{M_1}[T]$ and $G_{M_2}[T]$ either change sign or are nonpositive on $I\times I$.	
																			\item If $G_N[2\,T]$ is nonpositive on $J\times J$, then $G_N[T]$, $G_D[T]$, $G_{M_1}[T]$ and $G_{M_2}[T]$ either change sign or are nonpositive on $I\times I$.	
																			\item If $G_P[2\,T]$ is nonpositive on $J\times J$, then $G_N[T]$, $G_D[T]$, $G_{M_1}[T]$ and $G_{M_2}[T]$ either change sign or are nonpositive on $I\times I$.
																			\item If $G_P[4\,T]$ is nonpositive on $[0,4\,T]\times [0,4\,T]$, then $G_N[T]$, $G_D[T]$, $G_{M_1}[T]$ and $G_{M_2}[T]$ either change sign or are nonpositive on $I\times I$.
																			\item If $G_{M_2}[T]$ is nonpositive on $I\times I$, then $G_D[T]$ either changes sign or is nonpositive on $I\times I$.
																			\item If $G_D[2\,T]$ is nonpositive on $J\times J$, then $G_D[T]$ and $G_{M_2}[T]$ either change sign or are nonpositive on $I\times I$.
																		\end{itemize}
																	\end{corollary}

																	\section{Comparison Principles}
																	In this section we will use the connecting expressions for Green's functions obtained in Section 3 to compare the values that several Green's functions take point by point. It must be pointed out that, since the relations between the constant sign of the Green's functions are not as strong as for the case $n=1$, the results in this section will also be weaker (in some cases) than the ones obtained in \cite{CabCidSom2016}. 
																	
																	However, some results which could not be deduced for $n=1$ hold for $n>1$. This will be the case of Item 4 in Corollary \ref{c-comp-Neu-M1-2n} or Item 1 in Theorem \ref{t-comp-DM2-2n}, which do not make sense for the case $n=1$ because, in such a case, $G_D[2\,T]$ can never be nonnegative on $J\times J$.
																	
																	First, from \eqref{e-Green-DN-P-2n}, we obtain the following result.
																	\begin{corollary}\label{c-comp-Neu-Dir-2n}
																		\begin{enumerate}
																			\item 	If $G_P[2\,T]\geq 0$ on $J\times J$, then $G_N[T](t,s)\geq |G_D[T](t,s)|$ for all $(t,s)\in I\times I.$
																			\item If $G_P[2\,T]\le 0$ on $J\times J$, then
																			$G_N[T](t,s) \le -|G_D[T](t,s)|$  for all $(t,s)\in I\times I.$
																		\end{enumerate}
																	\end{corollary}
																	
																	The difference between this case and the particular one with $n=1$ and $a_1\equiv 0$ is that when $n=1$, the constant sign of $G_P[2\,T]$ ensures not only the constant sign of $G_N[T]$ but also that of $G_D[T]$. Thus, in such case, we would substitute $|G_D[T](t,s)|$ by $-G_D[T](t,s)$ in the inequalities given in previous corollary.
																	
																	As a consequence of previous corollary, we can compare the solutions of \eqref{e-N-2n} and \eqref{e-D-2n}:
																	
																	\begin{theorem}\label{t-comp-ND-2n}
																		Let $u_N$ be the unique solution of problem \eqref{e-N-2n} for $\sigma = \sigma_1$ and $u_D$ the unique solution of problem \eqref{e-D-2n} for $\sigma = \sigma_2$. Then
																		\begin{enumerate}
																			\item If $G_P[2\,T]\geq 0$ on $J\times J$ and $|\sigma_2(t)| \le \sigma_1(t)$ a.\,e. $t \in I$, then $|u_D(t)| \le u_N(t)$ for all $t \in I$.
																			\item If $G_P[2\,T]\le 0$ on $J\times J$ and $0\le \sigma_2(t) \le \sigma_1(t)$ a.\,e. $t \in I$, then $u_N(t)\le 0$ and $u_N(t) \le u_D(t)$ for all $t \in I$.
																			\item If $G_P[2\,T]\le 0$ on $J\times J$ and $\sigma_1(t) \le \sigma_2(t) \le 0$ a.\,e. $t \in I$, then $u_N(t)\ge 0$ and $u_D(t) \le u_N(t)$ for all $t \in I$.
																		\end{enumerate}
																	\end{theorem}
																	
																	\begin{proof}
																		\begin{enumerate}
																			\item Since $G_P[2\,T]\ge 0$ on $J\times J$ then, from Corollary \ref{c-comp-Neu-Dir-2n}, it holds that 
																			\begin{equation*}\begin{split}
																					|u_D(t)|&=\left|\int_{0}^{T} G_D[T](t,s)\,\sigma_2(s)\,\dif s\right| \le \int_{0}^{T} |G_D[T](t,s)|\,|\sigma_2(s)|\,\dif s \\ & \le \int_{0}^{T} G_N[T](t,s) \,\sigma_1(s)\, \dif s= u_N(t).
																				\end{split}\end{equation*}
																				\item Since $G_P[2\,T]\le 0$ on $J\times J$ then, from Corollary \ref{c-comp-Neu-Dir-2n}, since $\sigma_1(s)\ge 0$ a.\,e. $s\in I$, 
																				\[G_N[T](t,s)\,\sigma_1(s)\le -|G_D[T](t,s)|\,\sigma_1(s), \quad \forall\,(t,s)\in I\times I. \]
Moreover, from $\sigma_2(s) \le \sigma_1(s)$ and $\sigma_2(s)\ge 0$ a.\,e. $s\in I$, we deduce that
\[-|G_D[T](t,s)|\,\sigma_1(s)\le -|G_D[T](t,s)|\,\sigma_2(s)\le G_D[T](t,s)\,\sigma_2(s), \quad \forall\,(t,s)\in I\times I. \]
Therefore, for all $t\in I$, we have
\begin{equation*}\begin{split}
u_N(t)&=\int_{0}^{T} G_N[T](t,s)\,\sigma_1(s)\,\dif s \le \int_{0}^{T} -|G_D[T](t,s)|\,\sigma_1(s)\,\dif s \\ & \le \int_{0}^{T} -|G_D[T](t,s)| \,\sigma_2(s)\, \dif s \le \int_{0}^{T} G_D[T](t,s) \,\sigma_2(s)\, \dif s= u_D(t).
\end{split}\end{equation*}
Finally, the fact that $u_N\le 0$ on $I$ is deduced from $G_N[T]\le 0$ and $\sigma_1\ge 0$.
																					
																					\item Since $G_P[2\,T]\le 0$ on $J\times J$ then, from Corollary \ref{c-comp-Neu-Dir-2n}, it can be deduced that 
																					\[G_N[T](t,s)\le G_D[T](t,s) \ \text{ and } \ G_N[T](t,s)\le 0, \quad \forall\,(t,s)\in I\times I \]
and so, since $\sigma_2(s)\le 0$ and $\sigma_1(s) \le \sigma_2(s)$ a.\,e. $s\in I$, 
\[G_D[T](t,s)\,\sigma_2(s)\le G_N[T](t,s)\,\sigma_2(s) \le G_N[T](t,s)\,\sigma_1(s), \quad \forall\,(t,s)\in I\times I \]
Therefore,
\begin{equation*}\begin{split}
u_D(t)&=\int_{0}^{T} G_D[T](t,s)\,\sigma_2(s)\,\dif s  \le \int_{0}^{T} G_N[T](t,s) \,\sigma_1(s)\, \dif s= u_N(t).
\end{split}\end{equation*}
Finally, the fact that $u_N\ge 0$ on $I$ is deduced from $G_N[T]\le 0$ and $\sigma_1\le 0$.
\end{enumerate}
\end{proof}

																				So, the main difference with respect to the case $n=1$ is that when $n=1$ we are able to ensure the constant sign of function $u_D$, which does not happen in this case.
																				
																				Analogously, from \eqref{e-Green-NM1-N-2n} and \eqref{e-Green-DM2-D-2n}, the constant sign of $G_N[2\,T]$ and $G_D[2\,T]$ lets us deduce some point-by-point relation between various Green's functions. 
																				\begin{corollary}\label{c-comp-Neu-M1-2n}
																					\begin{enumerate}
																						\item If $G_N[2\,T]\geq 0$ on $J\times J$, then
																						$G_N[T](t,s)\geq |G_{M_1}[T](t,s)|$ for all $(t,s)\!\in I\times I.$
																						\item If $G_N[2\,T]\le 0$ on $J\times J$, then
																						$G_N[T](t,s) < -|G_{M_1}[T](t,s)|$ for all $(t,s)\in I\times I.$
																						\item If $G_D[2\,T]\leq 0$ on $J\times J$, then
																						$G_{M_2}[T](t,s) < -|G_{D}[T](t,s)|$ for all $(t,s)\in I\times I.$
																						\item If $G_D[2\,T]\geq 0$ on $J\times J$, then
																						$G_{M_2}[T](t,s) \geq |G_{D}[T](t,s)|$ for all $(t,s)\in I\times I.$
																					\end{enumerate}
																				\end{corollary}
																				
																				As a consequence of the previous corollary, we deduce the following comparison principles among the solutions of the corresponding problems. The arguments are similar to the ones used in the proof of Theorem~\ref{t-comp-ND-2n}.
																				\begin{theorem}\label{t-comp-NM1-2n}
																					Let $u_N$ be the unique solution of problem \eqref{e-N-2n} for $\sigma = \sigma_1$ and $u_{M_1}$ the unique solution of problem \eqref{e-M1-2n} for $\sigma = \sigma_2$. Then
																					\begin{enumerate}
																						\item If $G_N[2\,T]\geq 0$ on $J\times J$ and $|\sigma_2(t)| \le \sigma_1(t)$ a.\,e. $t \in I$, then $|u_{M_1}(t)| \le u_N(t)$ for all $t \in I$.
																						\item If $G_N[2\,T]\le 0$ on $J\times J$ and $0\le \sigma_2(t) \le \sigma_1(t)$ a.\,e. $t \in I$, then $u_N(t)\le 0$ and $u_N(t) \le u_{M_1}(t)$ for all $t \in I$.
																						\item If $G_N[2\,T]\le 0$ on $J\times J$ and $\sigma_1(t) \le \sigma_2(t) \le 0$ a.\,e. $t \in I$, then $u_N(t)\ge 0$ and $u_{M_1}(t) \le u_N(t)$ for all $t \in I$.
																					\end{enumerate}
																				\end{theorem}
																
											\begin{theorem}\label{t-comp-DM2-2n}
										Let $u_{M_2}$ be the unique solution of problem \eqref{e-M2-2n} for $\sigma = \sigma_1$ and $u_{D}$ the unique solution of problem \eqref{e-D-2n} for $\sigma = \sigma_2$. Then, it holds that:
						\begin{enumerate}
									\item If $G_D[2\,T]\geq 0$ on $J\times J$ and $|\sigma_2(t)| \le \sigma_1(t)$ a.\,e. $t \in I$, then $|u_D(t)| \le u_{M_2}(t)$ for all $t \in I$.
					\item If $G_D[2\,T]\le 0$ on $J\times J$ and $0\le \sigma_2(t) \le \sigma_1(t)$ a.\,e. $t \in I$, then $u_{M_2}(t)\le 0$ and $u_{M_2}(t) \le u_D(t)$ for all $t \in I$.
				\item If $G_D[2\,T]\le 0$ on $J\times J$ and $\sigma_1(t) \le \sigma_2(t) \le 0$ a.\,e. $t \in I$, then $u_{M_2}(t)\ge 0$ and $u_D(t) \le u_{M_2}(t)$ for all $t \in I$.
												\end{enumerate}
								\end{theorem}

\end{document}